\documentclass[11pt]{amsart}
\usepackage[utf8]{inputenc}
\usepackage{amsmath}
\usepackage{amsthm}
\usepackage{amsfonts}
\usepackage{amssymb}
\usepackage{enumitem}
\usepackage{esint}
\usepackage{color}

\title{Local vanishing mean oscillation}
\author[Butaev]{Almaz Butaev}
	\address{(A.B.) Department of Mathematical Sciences, P.O. Box 210025, University of Cincinnati, Cincinnati, OH 45221–0025, U.S.A}
	\email{butaevaz@ucmail.uc.edu}
	
	\author[Dafni]{Galia Dafni}
	\address{(G.D.) Concordia University, Department of Mathematics and Statistics, Montr\'{e}al, QC H3G 1M8, Canada}
	\email{galia.dafni@concordia.ca}

\thanks{G.D. was partially supported by the Natural Sciences and Engineering Research Council (NSERC) of Canada, and the Centre de recherches math\'e{}matiques (CRM)}


\newtheorem{theorem}{Theorem}
\newtheorem{lem}{Lemma}

\newtheorem{prop}{Proposition}
\newtheorem{definition}{Definition}
\newtheorem{example}{Example}

\newtheorem{manualtheoreminner}{Theorem}
\newenvironment{theoremA}[1]{%
\renewcommand\themanualtheoreminner{#1}%
\manualtheoreminner
}{\endmanualtheoreminner}

\newcommand{\ra}{\rightarrow}

\newcommand{\R}{{\mathbb R}}
\newcommand{\Rn}{{\R^n}}
\newcommand{\N}{{\mathbb N}}

\newcommand{\loc}{{\rm loc}}
\newcommand{\supp}{{\rm supp}}
\newcommand{\diam}{{\rm diam}}
\newcommand{\dist}{{\rm dist}}

\newcommand{\Loneloc}{{L^1_\loc}}
\newcommand{\Lploc}{{L^p_\loc}}
\newcommand{\Linfty}{{L^\infty}}
\newcommand{\BMO}{{\rm BMO}}
\newcommand{\bmo}{{\rm bmo}}
\newcommand{\VMO}{{\rm VMO}}
\newcommand{\vmo}{{\rm vmo}}
\newcommand{\CMO}{{\rm CMO}}
\newcommand{\cmo}{{\rm cmo}}
\newcommand{\lmo}{{\rm lmo}}
\newcommand{\Lip}{{\rm Lip}}

\newcommand{\UC}{{\rm UC}}

\newcommand{\Cinfty}{{C^\infty}}

\newcommand{\bmoc}{{\bmo_c}}
\newcommand{\bmol}{{\bmo_\lambda}}
\newcommand{\bmoO}{{\bmo(\Omega)}}
\newcommand{\vmoO}{{\vmo(\Omega)}}
\newcommand{\bmolo}{{\bmo_\lambda(\Omega)}}
\newcommand{\vmolo}{{\vmo_\lambda(\Omega)}}
\newcommand{\bmoloc}{{\bmo_{\lambda,c}(\Omega)}}
\newcommand{\bmolocobar}{{\bmo_{\lambda,c}(\Omegabar)}}
\newcommand{\bmoco}{{\bmoc(\Omega)}}

\newcommand{\bmolone}{{\bmo_{\lambda_1}(\Omega)}}
\newcommand{\bmoltwo}{{\bmo_{\lambda_2}(\Omega)}}

\newcommand{\cD}{{\mathcal D}}

\newcommand{\cN}{{\mathcal N}}

\newcommand{\ed}{{(\epsilon,\delta)}}
\newcommand{\led}{\lambda_{\epsilon, \delta}}
\newcommand{\dom}{{d_\Omega}}

\newcommand{\omegaO}{{\omega_\Omega}}

\newcommand{\Omegabar}{{\overline{\Omega}}}
\newcommand{\bOmega}{{\partial\Omega}}

\newcommand{\Omegat}{\mathring{\Omega}_{t}}
\newcommand{\Omegato}{\mathring{\Omega}_{t_0}}
\newcommand{\Omegal}{\mathring{\Omega}_{\lambda}}
\newcommand{\Omegalo}{\mathring{\Omega}_{\lambda/4}}
\newcommand{\Omegadel}{\mathring{\Omega}_{2\ell\sqrt{n}}}

\newcommand{\Tlam}{{T_\lambda}}
\newcommand{\tT}{\widetilde{\Tlam}}

\begin{document}

\maketitle

\begin{abstract}We consider various notions of vanishing mean oscillation on a (possibly unbounded) domain $\Omega \subset \Rn$, and prove an analogue of Sarason's theorem, giving sufficient conditions for the density of bounded Lipschitz functions in the nonhomogeneous space $\vmo(\Omega)$. We also study $\cmo(\Omega)$, the closure in $\bmoO$ of the continuous functions with compact support in $\Omega$.  Using these approximation results, we prove that there is a bounded extension from $\vmo(\Omega)$ and $\cmo(\Omega)$ to the corresponding spaces on $\Rn$, if and only if $\Omega$ is a locally uniform domain.
\end{abstract}

\section{Introduction}
In the theory of function spaces, it is useful to define ``local" versions of certain spaces.  For function spaces on $\Rn$, this can refer to functions defined globally which belong to a certain space when restricted to compact sets (such as $\Lploc(\Rn)$) or alternatively to spaces defined on a given domain $\Omega \subset \Rn$.  In the latter case one can then use ``local" to refer to behavior away from the boundary $\bOmega$.  When considering behavior up to the boundary, the notions of ``vanishing" at the boundary, approximation by smooth functions, and extension to $\Rn$ are interrelated and intimately connected with the geometry of the domain (for example, in the case of Sobolev spaces and Triebel-Lizorkin spaces, see  \cite{Brudnyi1, KoskelaRajalaZhang, KoskelaZhang, Jones2, Rychkov, SSS}).  

In the case of functions of bounded mean oscillation, there are various notions of ``local" and ``vanishing" in the literature.  The original definition of bounded mean oscillation by John and Nirenberg \cite{JN} was on a fixed cube $Q_0 \subset \Rn$: $f \in \BMO(Q_0)$ if
$$\sup\limits_{Q\subset Q_0} \fint_Q |f(x) - f_Q| dx < \infty,$$
where the supremum is over all parallel subcubes of $Q_0$, $|Q|$ denotes Lebesgue measure, and $f_Q:= \fint_Q f: = |Q|^{-1} \int_Q f$ is the average of $f$ on $Q$. 
 This can be extended to give a definition of $\BMO(\Omega)$ on any domain (or even open set) $\Omega \subset \Rn$ by taking the 
 supremum over all cubes  $Q \subset \Omega$ with sides parallel to the axes.  When $\Omega$ is connected, this supremum defines a norm modulo constants, and $\BMO(\Omega)$ is a Banach space.
 
A more refined measure of the mean oscillation of a function  $f \in \Loneloc(\Rn)$ is given by the {\em modulus of mean oscillation}
\begin{equation}
\label{eqn-modulus}
\omega(f, t):= \sup\limits_{\substack{\ell(Q)< t\\Q\subset\Rn}} \fint_Q |f(x) - f_Q| dx, \quad t > 0,
\end{equation}
where $\ell(Q)$ denotes the sidelength of $Q$. We then have $\|f\|_{\BMO(\Rn)}: = \sup_{t > 0}\omega(f,t)$.  A local version of BMO can be defined by fixing a finite $T > 0$ and considering locally integrable functions with
$\omega(f,T) < \infty$.
The set of such functions does not depend on the choice of $T$ and is strictly larger than $\BMO(\Rn)$ since it contains, for example, all uniformly continuous functions.
 
Another version of BMO which is called ``local" is the space $\bmo(\Rn)$ (not to be confused with what is known as ``little" BMO and has the same notation),  introduced by Goldberg \cite{Goldberg} as the dual of the local Hardy space $h^1(\Rn)$ and consisting of locally integrable functions $f$ satisfying
$$\|f\|_{\bmo(\Rn)} : = \omega(f, 1)  + \sup_{\ell(Q) \geq 1} |f|_Q < \infty,$$
where once more the supremum is taken over all cubes $Q \subset \Rn$ with sides parallel to the axes.   Here again the scale $1$ can be replaced by any finite $T$ without changing the collection of functions in $\bmo(\Rn)$, only affecting the norm, and we can also restrict the second term to the supremum of the averages of $|f|$ on cubes whose sidelength is exactly equal to $1$.  As sets of functions, $\bmo(\Rn)$ is strictly smaller than $\BMO(\Rn)$ (it does not contain $\log |x|$, for example) and should be considered as a nonhomogeneous version of BMO, not taken modulo constants.  Both $h^1$ and $\bmo$ are part of the scale of nonhomogeneous Triebel-Lizorkin spaces - see \cite[Theorem 1.7.1]{Triebel}.

The notion of vanishing mean oscillation was introduced by Sarason \cite{Sarason}.  The space $\VMO(\Rn)$ can be defined using either one of the two characterizations in the following theorem, which was proved in \cite{Sarason} for the case $n = 1$.
\begin{theoremA}{A}[Sarason] 
\label{thm-VMO}
For $f \in \BMO(\Rn)$, 
\begin{equation}
\label{eq-VMO}
\displaystyle{\lim_{t \ra 0^+} \omega(f, t) = 0}
\end{equation}
if and only if
$f \in \overline{\UC(\Rn) \cap \BMO(\Rn)}$, the closure of the uniformly continuous functions in $\BMO$. 
\end{theoremA}
A smaller space which is sometimes also called $\VMO(\Rn)$ (see \cite{CoifmanWeiss}),  and serves as a predual to the Hardy space $H^1(\Rn)$, is the closure in $\BMO(\Rn)$ of the continuous functions with compact support (or equivalently the $\Cinfty$ functions with compact support).  We will denote this space by $\CMO(\Rn)$ for ``continuous mean oscillation", following Neri \cite{Neri}.  As stated  in \cite{Neri} and proved by  Uchiyama in \cite{Uchiyama},  in addition to \eqref{eq-VMO}, functions  in $\CMO(\Rn)$ also satisfy vanishing mean oscillation conditions as the size of the cube increases to $\infty$ and as the cube itself goes to $\infty$.  Recently subspaces between $\CMO(\Rn)$ and $\VMO(\Rn)$ were considered in \cite{TXYY, TorresXue}.

The nonhomogeneous versions of the spaces $\VMO(\Rn)$ and $\CMO(\Rn)$, denoted $\vmo(\Rn)$ and $\cmo(\Rn)$, are the corresponding subspaces of $\bmo(\Rn)$, and the vanishing mean oscillation conditions characterizing the latter were given in  \cite{Bourdaud, Dafni} (see Proposition~\ref{prop-cmo}).  Bourdaud's paper  \cite{Bourdaud} contains extensive coverage of the properties of  $\BMO(\Rn)$ and $\bmo(\Rn)$ (treating it modulo constants as a subspace of $\BMO$) as well as their vanishing subspaces.

The focus of our work are the versions of these spaces on a domain $\Omega \subset \Rn$, and the corresponding approximation and extension results.   The definition of the modulus of oscillation can be adapted by restricting the cubes to lie inside the domain, namely
\begin{equation}
\label{eqn-modulus-Omega}
\omegaO(f, t):= \sup\limits_{\substack{\ell(Q)< t\\Q\subset\Omega}} \fint_Q |f(x) - f_Q| dx, \quad t > 0,
\end{equation}
and $\BMO(\Omega)$ defined to consist of those $f \in \Loneloc(\Omega)$ with $\sup_{t > 0}\omegaO(f,t) < \infty$.  The question of the definition of $\VMO(\Omega)$ is more delicate: for which domains does a version of Sarason's theorem hold?  For a bounded domain $\Omega$, Brezis and Nirenberg \cite{BN2} give many results on $\VMO(\Omega)$, including a strong version of Sarason's result, which they attribute to Jones, not only identifying the subspace of $\BMO(\Omega)$ consisting of functions with vanishing mean oscillation with the closure of the uniformly continuous functions, but also with the closure of the continuous, or smooth, functions with compact support in $\Omega$.

Jones \cite{Jones} showed that 
there is a bounded linear extension from $\BMO(\Omega)$ to $\BMO(\Rn)$ if and only if $\Omega$ is a uniform domain.  In \cite{BD1}, we composed Jones' extension operator with an averaging operator on the complement of $\Omegabar$ to obtain an operator simultaneously extending $\BMO$, $\VMO$ and $\CMO$, as well as Lipschitz functions, on a uniform domain $\Omega$, and also characterized such domains in terms of the existence of a bounded extension from $\CMO(\Omega)$ to $\BMO(\Rn)$.  As a corollary, we obtained a version of Sarason's theorem for these domains.

Turning to the nonhomogeneous case, in  \cite{BD2} we proved a version of Jones' theorem for $\bmo$, identifying the extension domains with locally uniform domains, which we in turn showed are equivalent to the $\ed$-domains used in Jones' extension results for Sobolev spaces in \cite{Jones2}.  
To define $\bmo$ on $\Omega$ we need to fix a scale $\lambda$.  We say  $f \in \bmolo$ if $f$ is integrable on every cube $Q \subset \Omega$ and 
\begin{equation}
	\label{def-bmolo}
	\|f\|_\bmolo :=\omegaO(f, \lambda)  + \sup_{Q\subset \Omega, \ell(Q) \geq \lambda} |f|_Q < \infty.
\end{equation}
Due to the extension theorem in \cite{BD2}, if $\Omega$ is an $\ed$-domain, there is a natural scale $\led$ such that membership in $\bmolo$ is independent of $\lambda$ provided $\lambda \leq \led$, and we can define $\bmoO$ to be $\bmo_{\led}(\Omega)$.  As pointed out above, in the case $\Omega = \Rn$, one can take $\lambda$ to be any finite positive number.

In the special case when $\Omega$ is a bounded domain, $\BMO(\Omega)$ and $\bmolo$ coincide for $\lambda$ sufficiently small, provided we consider them both modulo constants or fix the average on some large cube in $\Omega$ to be zero, say.   In such a case Jones' extension in \cite{Jones} will vanish on all cubes sufficiently far away from $\Omega$.
In \cite{BD0},  
we gave an extension operator for $\VMO(\Omega)$ on a bounded uniform domain which does not use averaging and preserves the property of the Jones extension that the values of the extended function on a Whitney cube of $\Rn \setminus \Omegabar$ are completely determined by the average of the original function on a matching cube inside $\Omega$.  

Turning to unbounded domains, in the present paper we prove an extension theorem for the vanishing mean oscillation subspaces of $\bmo(\Omega)$, or equivalently a nonhomogeneous version of the VMO extension theorem in \cite{BD1}.  Unlike in \cite{BD1}, where we used the extension result to prove the analogue of Sarason's theorem, here we first prove the approximation theorem and then obtain the extension as a corollary. It is natural to state the approximation in terms of Lipschitz functions since, as shown in \cite{BD1},  the averaging process used by Sarason yields this strong form of uniform continuity, and it is used for both the approximation and the extension arguments.  As in the setting of metric measure spaces, Lipschitz functions are the ``smooth" functions in this context.

\begin{theorem}
\label{thm-approxdomain}
Suppose $\Omega \subset \Rn$ is a domain and $\lambda$ is such that for all $0 < \lambda' \leq \lambda$, $\bmo_{\lambda'}(\Omega) = \bmolo$ as sets.  Then
$$\{f \in \bmolo: \lim_{t \ra 0^+} \omegaO(f, t) = 0\}= \overline{\UC(\Omega) \cap  \Linfty(\Omega)} =  \overline{\Lip_b(\Omega)} =  \overline{\Lip_{b,0}(\Omega)},$$ 
where the closures are in the $\bmolo$ norm, $\Lip_b(\Omega)$ denotes the bounded Lipschitz functions in $\Omega$, equipped with the nonhomogeneous  norm
$$\|f\|_{\Lip_b}:= \|f\|_\infty + \sup_{x \neq y} \frac{|f(x) - f(y)|}{|x - y|},$$
and $\Lip_{b,0}(\Omega)$ consists of those $f \in \Lip_b(\Omega)$  with $\dist(\supp(f), \bOmega) > 0$.
\end{theorem}

As pointed out above, for an $\ed$-domain $\Omega$, the hypothesis of the theorem holds for $\lambda = \led$, and as a result we can not only unambiguously define $\bmoO$ but also $\vmo(\Omega)$. We define $\cmo(\Omega)$ to be the closure in $\bmoO$ of $C_c(\Omega)$, the continuous functions with compact support in $\Omega$.  This can be identified (see Proposition~\ref{prop-vmo-cmo}) with the space of functions in $\vmoO$ which vanish at infinity.  Note that both $\Lip_b(\Omega)$ and $C_c(\Omega)$ are continuously embedded in $\bmolo$ since $\|f\|_\bmolo \leq \|f\|_\infty$.

Using these definitions, we can state our extension result.

\begin{theorem}
\label{thm2}
Let $\Omega \subset \Rn$ be an $\ed$-domain.  Then there exists a linear extension operator $T$ such that 
\begin{itemize}
    \item[\rm{(i)}] $T:\bmoO \to \bmo(\Rn)$ is bounded;
    \item[\rm{(ii)}] $T:\vmo(\Omega) \to \vmo(\Rn)$ is bounded;
    \item[\rm{(iii)}]  $T:\cmo(\Omega) \to \cmo(\Rn)$ is bounded; 
    \item[\rm{(iv)}] $T:\Lip_b(\Omega) \to \Lip_b(\Rn)$ is bounded.
\end{itemize}
Boundedness in (i)-(iv) refers to the $\bmo$ norm while in (v) the boundedness is with respect to the norm $\|f\|_{\Lip_b}$.
\end{theorem}

Note that while Theorem~\ref{thm-approxdomain} holds under weaker assumptions on the domain (see Section~\ref{sec-examples} for examples which are not $\ed$ domains), it is not possible to weaken the assumptions in Theorem~\ref{thm2}, as the converse holds.  This follows from the results in \cite{BD2} by observing that the functions used in the proof of Theorem 3.1 there, constructed from the quasihyperbolic metric, are continuous with compact support in $\Omega$.  Thus we can state the converse under the weakest hypotheses.

\begin{theorem}[\cite{BD2}]
\label{thm3}
If $\Omega \subset \Rn$ is a domain and for some $\lambda > 0$ there is an extension operator $T: C_c(\Omega)\to \BMO(\Rn)$ such that for some $C$ and all $f \in C_c(\Omega)$,
$$\|Tf\|_{\BMO(\Rn)}\leq C\|f\|_\bmolo,$$
then $\Omega$ is an $\ed$-domain. 
\end{theorem}

We start, in Section~\ref{sec-approximation}, with some results on approximation by smooth and compactly supported functions in $\bmo(\Rn)$, and clarify the notion of ``vanishing at infinity" in this context, answering a question posed by Bourdaud in \cite{Bourdaud}.  Section~\ref{sec-approximation2} addresses the same questions on a domain $\Omega$, culminating in the proof of Theorem~\ref{thm-approxdomain} in Subsection~\ref{sec-approx-vmo}.  The proof of Theorem~\ref{thm2} can be found in Section~\ref{sec-extension1}.  Finally, Section~\ref{sec-examples} provides some examples and counterexamples to illustrate the results.

\section{Approximation and vanishing mean oscillation in $\Rn$}
\label{sec-approximation}

The approximation of functions in $\VMO(\Rn)$ by Lipschitz functions can be proved by the same technique use by Sarason on $\R$, namely first taking  the averages of the function on a sufficiently fine grid of cubes, and then smoothing out the resulting step function by convolution with a mollifier.   Even when convolving with the normalized characteristic function of a ball, this produces a Lipschitz function (see \cite{BD1}).   Averaging functions of vanishing mean oscillation in order to get smooth functions is a standard technique, and can be used to show that if the modulus of oscillation is sufficiently rapidly decreasing, the functions themselves are smooth, as originally shown by Campanato \cite{Campanato}, Meyers \cite{Meyers} and Spanne \cite{Spanne} (see also \cite{Sarason2}).  For $\BMO$ functions without vanishing mean oscillation, on a bounded domain, averaging results in Lipschitz up to a log factor, as shown in \cite[Lemma B.9]{BN1}.
 
For $f \in \bmo(\Rn)$,   the functions resulting from this process, for a given grid size, are bounded, as we can control the average of $f$ on a cube $Q$ with $\ell(Q) < 1$ by taking a chain $Q = Q_0 \subset Q_1 \subset \ldots Q_k$ with $\ell(Q_i) = 2\ell(Q_{i-1})$, $1 \leq i \leq k -1$, and $1 = \ell(Q_k) \leq 2\ell(Q_{k-1}) < 2$, and using the standard estimate
\begin{equation}
\label{eq-logbmo}
|f|_Q \leq \sum_{i = 1}^k ||f|_{Q_{i-1}} - |f|_{Q_i}| + |f|_{Q_k} \leq 2^{n+1} k \sup_{1 \leq i \leq k} \fint_{Q_i} |f - f_{Q_i}|+ |f|_{Q_k}  \lesssim \log\Big(\frac 2{\ell(Q)}\Big)\|f\|_\bmo.
\end{equation}
Here and throughout the paper we use $a \lesssim b$ to denote the existence of a constant $C$ (usually only depending on the dimension) such that $a \leq Cb$.

This gives us the result of Bourdaud \cite[Th\'eor\`eme 1]{Bourdaud} identifying $\vmo(\Rn)$ with the closure of the bounded uniformly continuous functions in $\bmo(\Rn)$.  

\begin{theoremA}{B}[\cite{Bourdaud}]  
\label{thm-bourdaud}
A function $f \in \bmo(\Rn)$ satisfies \eqref{eq-VMO} if and only if it can be approximated in the $\bmo$ norm by bounded uniformly continuous functions.  
\end{theoremA} 
Bourdaud also proves (see \cite[Th\'eor\`eme 4]{Bourdaud}) that functions in $\vmo(\Rn)$ can be approximated by $C^\infty$ functions, but such functions do not have uniform (on $\Rn$) bounds on their derivatives.  What is possible, as noted above, is to strengthen uniform continuity to Lipschitz continuity.

One can get more smoothness when there is {\em vanishing at infinity}.  The following characterization of $\cmo(\Rn)$ captured this notion in the $\bmo$ sense.

\begin{prop}[\cite{Bourdaud,Dafni}]
\label{prop-cmo} 
For $f \in \bmo(\Rn)$, the following are equivalent:
\begin{enumerate}
\item $f \in \cmo(\Rn)$, the closure in $\bmo(\Rn)$ of $C_c(\Rn)$;
\item $f$ satisfies \eqref{eq-VMO} together with $\displaystyle{\lim_{\beta \ra \infty} \sup\{|f|_Q : \dist(Q, 0) > \beta, \ell(Q)\geq1\} =0}$;
\item  $f$ satisfies \eqref{eq-VMO} together with $$\limsup_{x \ra \infty} \sup\left\{\fint_Q |f - f_Q| : \mbox{center of } Q = x, \ell(Q) \leq 1\right\} = 0=\limsup_{x \ra \infty}|f|_{Q_0 + x},$$
 where $Q_0 = [0,1]^n$.
\end{enumerate}
\end{prop}

The equivalence of the first two conditions was proved by the second author in \cite{Dafni}, independently of the work of Bourdaud \cite{Bourdaud}, who proved the equivalence of the first and third condition.   Bourdaud defines $\cmo(\Rn)$ as the closure in $\bmo(\Rn)$ of $\cD(\Rn)$, the set of smooth functions with compact support, so the identifications above include the approximation by Lipschitz functions of compact support.   Note that the fact that the vanishing at infinity of the averages over large cubes in the second condition is sufficient to give the vanishing at infinity of the oscillation on small cubes in the third condition can be seen by combining Sarason's VMO condition  \eqref{eq-VMO} with the standard log estimate \eqref{eq-logbmo}, but may fail when not working on all of $\Rn$ (see Example~\ref{example2}).   

Bourdaud, in the same paper \cite[p.\ 1217 (1)]{Bourdaud}, poses as a question for further study the characterization of the notion of vanishing at infinity for functions in $\bmo(\Rn)$.   We answer this question by giving an analogue of Proposition~\ref{prop-cmo}  for functions which do not necessarily satisfy the VMO condition  \eqref{eq-VMO} . 

\begin{definition}
We denote by $\bmoc(\Rn)$ the subspace of functions $f$ in $\bmo(\Rn)$ which have compact support.  We say $f \in \bmo(\Rn)$ {\em vanishes at infinity} if
\begin{equation}
\label{vanishing_at_infinity}
\lim_{R \ra \infty}\|f\|_{\bmo(\Rn \setminus \overline{B(0, R)})} = 0,
\end{equation}
where the $\bmo$ norm is taken here in the sense of Definition~\ref{def-bmolo} with $\lambda = 1$.
\end{definition}

\begin{prop}
\label{prop-bmo_0}
For $f \in \bmo(\Rn)$, the following are equivalent:
\begin{enumerate}
\item[(i)] $f$ is in the closure of $\bmoc(\Rn)$  in $\bmo(\Rn)$;
\item[(ii)] $f$ vanishes at infinity;
\item[(iii)] $\displaystyle{\lim_{\beta \ra \infty} \gamma(f,\beta) = 0}$, where
$$\gamma(f,\beta) := \sup_{\dist(Q, 0) > \beta, \ell(Q)< 1} \fint_Q |f - f_Q|  + \sup_{\dist(Q, 0) > \beta,  \ell(Q)=1}|f|_Q;$$
and
\item[(iv)] 
\begin{equation*}
\lim_{\beta \ra \infty} \left(\sup_{\dist(Q, 0) > \beta, \ell(Q)< 1} \fint_Q |f - f_Q|  + \sup_{\dist(Q, 0) > \beta, \ell(Q)\leq  1}|f|_Q\ell(Q)\right) = 0.
\end{equation*}
\end{enumerate}
\end{prop}

\begin{proof}
Condition \eqref{vanishing_at_infinity} is satisfied for any function of compact support, and is preserved when taking limits in the $\|\cdot\|_{\bmo(\Rn)}$ norm, so (i) $\implies$ (ii).  Moreover, $\gamma(f,\beta) \leq
\|f\|_{\bmo(\Rn \setminus \overline{B(0, \beta)})}$ so (ii) $\implies$  (iii).   

To show (iii) $\implies$  (iv), it suffices to bound $|f|_Q\ell(Q)$ for a cube $Q$ with $\dist(Q, 0) > \beta$, and $\ell(Q) < 1$.   Proceeding as in \eqref{eq-logbmo} with  $Q \subset Q_k \subset \Rn \setminus \overline{B(0,\beta)}$ and $\ell(Q_k)= 1$, we have
$$|f|_Q\lesssim \log(2/\ell(Q))\;\gamma(f, \beta) \lesssim \ell(Q)^{-1} \gamma(f, \beta).$$
This shows  the quantity in parenthesis in (iv) is controlled by  $\gamma(f, \beta)$.

We will now complete  the circle of equivalences by proving that (iv) $\implies$ (i).
Assume $f$ satisfies condition (iv) (and hence also (iii)). 
First note that if $P$ is a cube whose sidelength is greater than $1$, then taking a cube $P'\supset P$ with $\ell(P')$ an integer and $\frac{\ell(P')}{\ell(P)} < 2$, and writing $P'$ as the union of cubes $Q_i$  of sidelength $1$
and pairwise disjoint interiors, we have
\begin{eqnarray*}
\fint_P|f| & \leq & \frac{1}{|P|} \left(\sum_{\dist(Q_i,0) \leq \beta} \int_{P \cap Q_i} |f| +   \sum_{\dist(Q_i,0) > \beta} |Q_i||f|_{Q_i}\right)\\
& \leq &  \frac{1}{|P|}\int_{P \cap B(0, \beta + \sqrt{n})} |f| +2^n \sup_{\dist(Q, 0) > \beta, \ell(Q)=1}|f|_Q.
\end{eqnarray*}
Thus (iii) implies the seemingly stronger vanishing condition for averages over large cubes,
\begin{equation}
\label{eq-largecubes}
\lim_{\beta \ra \infty} \left[ \sup_{\dist(Q, 0) > \beta, \ell(Q)\geq 1}|f|_Q\ + \sup_{\ell(Q) \geq \beta}|f|_Q\right] =0,
\end{equation}
and in particular (iii) $\implies$ (ii).

Fix $k \in \N$ and let
$f_k = \psi_k f$, where $\psi_k$ is a Lipschitz function which is equal to $1$ on $B(0, k)$ and  to $0$ outside $B(0, 2k)$, with $0 \leq \psi_k \leq 1$ and $\|\psi_k\|_\Lip \leq k^{-1}$.  Since the support of $f_k$ is compact, we just need to show  $\|f - f_k\|_{\bmo(\Rn)} \ra 0$ as $k \ra \infty$.  

Write $g_k = 1 - \psi_k$ so $f - f_k = fg_k$.  For any cube $Q$,  $|fg_k|_Q \leq |f|_Q$, 
so by (iii) and \eqref{eq-largecubes} we know the averages of $fg_k$ decay as the size of the cube increases.  Hence we may assume that $\diam(Q) < k/2$;  
from this, knowing $f g_k = 0$ on $B(0,k)$, we can restrict to  cubes with $\dist(Q, 0) \geq k/2$, and therefore, again by  \eqref{eq-largecubes}, to $\ell(Q) < 1$.

In order to bound the oscillation of $fg_k$ over such cubes $Q$, setting $c_Q =  f_Q(g_k)_Q$, we have
\begin{eqnarray}
\label{eq-Leibniz}
 \fint_Q |fg_k -c_Q|   & \leq   &
\|g_k\|_\infty \fint_Q |f- f_Q|   + 2|f_Q| \fint_Q |g_k - (g_k)_Q|\\
\nonumber 
& \leq   &
\fint_Q |f- f_Q|   + 2|f_Q|\; \diam(Q_k)\; \|g_k\|_\Lip\\
\nonumber & \lesssim &  \sup_{\dist(Q, 0) \geq k/2, \ell(Q)< 1} \left(\fint_Q |f - f_Q|  +|f_Q|\ell(Q)k^{-1}\right).
\end{eqnarray}
By (iv), this supremum will go to zero as $k \ra \infty$.
\end{proof}

In \cite[Proposition 3.3]{BonamiFeuto}, estimate  \eqref{eq-Leibniz} above and the logarithmic estimate \eqref{eq-logbmo} are used to bound the $\bmo$ norm of the product of a function in $\bmo(\Rn)$ with a bounded function in $\lmo(\Rn)$. However, as can be seen in the proof, the logarithmic estimate  \eqref{eq-logbmo} itself is not necessary to prove the approximation (i), and may not hold in an arbitrary domain (see Examples~\ref{example1} and \ref{example2}).

\section{Approximation and vanishing mean oscillation on a domain}
\label{sec-approximation2}

Our goal is to study the approximation by Lipschitz functions and functions of compact support and prove Theorem~\ref{thm-approxdomain}, which is the analogue of Theorems~\ref{thm-VMO} and \ref{thm-bourdaud}, as well as analogues of Propositions~\ref{prop-cmo} and \ref{prop-bmo_0} in the nonhomogeneous space $\bmo$ on a domain $\Omega$.  Such results are not only of interest in themselves but will lead to a simple proof of the extension theorem, Theorem~\ref{thm2}.

\subsection{Approximation in $\bmolo$}
We start with the case where we do not assume any Sarason-type vanishing mean oscillation condition.  We will need the following notation and terminology.
\begin{definition}
\label{def-vanishing_at}
We denote by $\bmoloc$ the subspace of functions $f$ in $\bmolo$ which have compact support in $\Omega$.  We say $f \in \bmolo$ {\em vanishes at infinity}  if
\begin{equation}
\label{vanishing_at_infinity_O}
\lim_{R \ra \infty}\|f\|_{\bmol(\Omega \setminus \overline{B(0, R)})} = 0.
\end{equation}
We say  $f \in \bmolo$ {\em vanishes at the boundary} if 
\begin{equation}
\label{vanishing_at_bO}
\lim_{t \ra 0} \|f\|_{\bmol(\Omega \setminus \Omegat)} = 0
\end{equation}
where we define
$$\Omegat:= \{x\in \Omega: \dom(x) \geq t\}, \quad \dom(x):=  \dist(x, \bOmega).$$
\end{definition}
To avoid confusion, it should be pointed out that in \cite{BD2}, $\lambda$ was replaced by $\lambda/4$ in the definition of $\Omegal$, while in \cite{Bourdaud} the notation $\Omega_t$ was used completely differently, for the set $\Rn \setminus\overline{B(0,t)}$.

Note that for an arbitrary domain $\Omega$ and $\lambda > 0$, if $f \in \bmolo$ with $\dist(\supp(f),\bOmega)) > 0$ then $\|f\|_{\bmol(\Omega \setminus \Omegat)} = 0$ for all sufficiently small $t$.  Since $\|f\|_{\bmol(\Omega \setminus \Omegat)}  \leq  \|f\|_\bmolo$, we have that the limit of such functions vanishes at the boundary.  

Similarly, if $f \in \bmolo$ has compact support in $\Omegabar$ (we use the notation $\bmolocobar$ to denote such functions) then $f$ vanishes outside a bounded set and $\|f\|_{\bmol(\Omega \setminus \overline{B(0, R)})}   = 0$ for all sufficiently large $R$.  This again is dominated by the $\bmolo$ norm, so the limit of such functions vanishes at infinity.  
 
In order to get the reverse implications, we need to assume,  for functions in $\bmolo$, some control of the averages over small cubes in terms of the sidelength of the cubes.  This is accomplished by the following two lemmas, stated and proved for general domains.

\begin{lem}
\label{lem-vanishing_at_infinity}
Let $\Omega$ be a domain, $\lambda > 0$ and $f \in \bmolo$.  Suppose $f$ vanishes at infinity and
\begin{equation}
\label{eq-1/t}
\sup_{Q \subset \Omega, \dist(Q, 0) > \beta, \ell(Q)< \lambda}|f_Q| \ell(Q) < \infty
\end{equation}
for $\beta$ sufficiently large.
Then $f$ is in the closure of $\bmolocobar$  in $\bmolo$.
\end{lem}

\begin{proof}
We proceed as in the proof of Proposition~\ref{prop-bmo_0}.  The vanishing at infinity immediately implies $\displaystyle{\lim_{\beta \ra \infty} \gamma_\Omega(f,\beta) = 0}$, where
$$\gamma_\Omega(f,\beta) := \sup_{Q \subset \Omega, \dist(Q, 0) > \beta, \ell(Q)< \lambda} \fint_Q |f - f_Q|  + \sup_{Q \subset \Omega, \dist(Q, 0) > \beta, \ell(Q)=\lambda}|f|_Q.$$
By restricting the arguments to cubes lying inside $\Omega$, we get from this an analogue of \eqref{eq-largecubes}, namely
\begin{equation}
\label{eq-largecubes-Omega}
\lim_{\beta \ra \infty} \left[ \sup_{Q \subset \Omega, \dist(Q, 0) > \beta, \ell(Q)\geq \lambda}|f|_Q\ + \sup_{Q \subset \Omega, \ell(Q) \geq \beta}|f|_Q\right] =0.
\end{equation}
Then we define the functions $f_k = f \psi_k$, noting that the supports of these functions are bounded and are therefore compact subsets of $\Omegabar$.  We continue with the arguments in the proof of Proposition~\ref{prop-bmo_0}, restricted to cubes $Q \subset \Omega$.  In the final step, we apply \eqref{eq-Leibniz} to estimate the oscillation of $f - f_k$ on cubes $Q \subset \Omega$ with $\dist(Q, 0) \geq k/2$ and $\ell(Q) < \lambda$.  As in the final step of that proof, all that is needed for this is the vanishing of $\gamma_\Omega(f,\beta)$ and \eqref{eq-1/t}.
\end{proof}

To obtain the analogous result for the vanishing at the boundary, we need to refine our hypothesis.

\begin{lem}
\label{lem-vanishing_at_boundary}
Let $\Omega$ be a domain, $\lambda > 0$ and $f \in \bmolo$.  Suppose $f$ vanishes at the boundary and there exists a 
function  $\varphi_\lambda: (0,\lambda/2) \ra  (0,\infty)$ with $\varphi_\lambda(t) =  1$ for $t \geq \lambda/4$, $t\varphi_\lambda(t)$ monotone increasing,  $\int_0^{\lambda /4}\frac{dt}{t\varphi_\lambda(t)} = \infty$, and such that
\begin{equation}
\label{eq-varphi}
\sup_{\substack{2Q \subset \Omega\\ \dist(Q, \bOmega) < \lambda/4}}|f_Q| \lesssim \varphi_{\lambda}(\ell(Q)).
\end{equation}
Then there exists a sequence $\{f_j\}$ converging to $f$  in  $\bmolo$, with $\dist(\supp(f_j),\bOmega)) > 0$ for each $j$.
\end{lem}

\begin{proof}  Set $f_j = fh_j$, where $h_j$ are modified versions of the auxiliary functions introduced in the proof of \cite[Theorem 1]{BN2}:
$$h_j(x) := \Big(1 - \frac 1 j \int_{\dom(x)}^{\lambda/4} \frac{dt}{t\varphi_\lambda(t)}\Big)_+.$$
Here the notation $F_+$ denotes $\max(F,0)$.
For $x \in \Omegalo$ we have that $h_j(x) = 1$, giving $f_j(x)  = f(x)$.  Moreover, by the hypothesis on $\varphi_\lambda$, there is a sequence of positive numbers $\alpha_j$ converging to zero such that $\dom(x) \leq \alpha_j \iff \int_{\dom(x)}^{\lambda/4} \frac{dt}{t\varphi_\lambda(t)} \geq j \iff h_j(x) = 0$, so $\dist(\supp(f_j),\bOmega))\geq \alpha_j$.  That $f_j \in \bmolo$ and $f_j \ra f$ will follow by estimating $\|f - f_j\|_\bmolo$.

Let $g = f - f_j$.  To estimate $\|g\|_\bmolo$, we use a ``local-to-global" property - see \cite[Theorem A1.1]{BN2} applied to cubes (balls in the $\ell^\infty$ norm) and \cite[Lemma 3.5]{BD2}:
$$\|g\|_\bmolo \lesssim \sup_{2Q \subset \Omega} \fint_Q |g - g_Q|  + \|g\|_{\Linfty(\Omegalo)}.$$
Since $g = 0$ on $\Omegalo$, it remains to bound the oscillation of $g$ over cubes $Q$ with $2Q \subset \Omega$ and $\dist(Q, \bOmega) < \lambda/4$, which means $\ell(Q) < \lambda/2$. 

As in the proof of Proposition~\ref{prop-bmo_0}, we can control the oscillation of $g = f(1-h_j)$ over $Q$ by 
\begin{equation}
\label{eqn-split}
\|1 - h_j\|_{\Linfty(Q)} \fint_Q |f - f_Q|   + |f_Q| \fint_Q |h_j - (h_j)_Q|.
\end{equation}
For a given $\eta > 0$, take $t_0$ sufficiently small so that $\|f\|_{\bmol(\Omega \setminus \Omegato)} < \eta/2$.  Then if $Q \cap \Omegato = \varnothing$ we have $\fint_Q |f(x) - f_Q| dx < \eta/2$ and 
$\|1 - h_j\|_{\Linfty(Q)}  \leq 1$.  

When $Q \cap \Omegato$ is nonempty, the condition $2Q \subset \Omega$ forces $\dist(Q, \bOmega) \geq t_0/(1 + \sqrt{n})$.  Hence for $j$ sufficiently large so that $\alpha_j < t_0/(1 + \sqrt{n})$, we have
$$|1 - h_j(x)|= \frac 1 j \int_{\dom(x)}^{\lambda/4} \frac{dt}{t\varphi_\lambda(t)}  \leq  \frac 1 j \int_{\frac{t_0}{1 + \sqrt{n}}}^{\lambda/4} \frac{dt}{t\varphi_\lambda(t)}, \quad \forall\; x \in Q.$$
Since the integral on the right depends only on $\lambda$ and $t_0$, bounding the oscillation of $f$ on $Q$ by $\|f\|_\bmolo$, we can make the first term in \eqref{eqn-split} smaller than $\eta/2$ by taking $j$ sufficiently large.  

For the second term in \eqref{eqn-split}, letting $\Phi_\lambda(x) = \int_{\dom(x)}^{\lambda/4} \frac{dt}{t\varphi_\lambda(t)}$ and noting that truncations do not increase oscillation and that $\dom$ is a $1$-Lipschitz function, we can proceed as in the proof of \cite[Lemma 4]{BN2}:
\begin{eqnarray}
\nonumber
|f_Q| \fint_Q |h_j(x) - (h_j)_Q| dx & \leq & \frac {|f_Q|} j \fint_Q \fint_Q |\Phi_{\lambda}(x) - \Phi_{\lambda}(y)| dxdy\\
\nonumber
& \lesssim & \frac{\varphi_\lambda(\ell(Q))} j \frac{\diam(Q)}{\inf_{x \in Q}\dom(x)\varphi_\lambda(\dom(x))} \\
\nonumber
& \lesssim &  \frac 1 j \frac{\ell(Q)\varphi_\lambda(\ell(Q))}{\dist(Q,\bOmega) \varphi_{\lambda}(\dist(Q,\bOmega))}\\
\label{eq-lmo}
& \lesssim &  \frac {1} j.
\end{eqnarray}
Here we have used \eqref{eq-varphi} and the fact that $2Q \subset \Omega$ implies $\ell(Q) \leq \dist(Q,\bOmega)$, as well as the fact that  $t\varphi_\lambda(t)$ is an increasing function.
Thus the second term in \eqref{eqn-split} can be made small for $j$ sufficiently large.
\end{proof}

The choice of $\varphi_\lambda(t)= \frac 1 t$, corresponding to the bound in \eqref{eq-1/t}, does not satisfy the conditions in Lemma~\ref{lem-vanishing_at_boundary}.  At the other extreme, putting $\varphi_\lambda(t) = 1$ in the hypotheses of Lemma~\ref{lem-vanishing_at_boundary} corresponds to $f$ being a bounded function, which is the special case shown in the proof of Theorem 1 in \cite{BN2} for a bounded domain, and in Lemma 3 of \cite{BD1} for a general domain.  In these results the vanishing at the boundary follows immediately from assuming the Sarason-type VMO condition restricted to $\Omega$, as will be seen in Section~\ref{sec-approx-vmo}.

We now introduce certain geometric conditions on the domain which will allow us to apply Lemmas~\ref{lem-vanishing_at_infinity} and \ref{lem-vanishing_at_boundary} to all functions in $\bmolo$.  Let us recall the definition of an $\ed$-domain by Jones \cite{Jones2}.
\begin{definition}
Given $\epsilon\in (0,1]$ and $\delta>0$, an $\ed$-domain is a domain $\Omega$ such that   every pair of points $x,y$ in $\Omega$ with $|x-y|< \delta$
may be joined by  a rectifiable curve $\gamma$ lying in $\Omega$, with
\begin{equation*}
	\mbox{arclength}(\gamma) \leq \epsilon^{-1}|x-y|,
\end{equation*}
and such that for any point $z$ on $\gamma$,
\begin{equation*}
	\dist(z, \bOmega) \geq \epsilon\frac{|z-x||z-y|}{|x-y|} .
\end{equation*} 
\end{definition} 

For such domains we have the following analogue of \eqref{eq-logbmo}, which can be shown directly by a combination of Proposition 4.9  and  Lemma 4.4 in \cite{BD2}, or by extending $f$ to $\bmol(\Rn)$ and then using \eqref{eq-logbmo}.
  
\begin{lem}[\cite{BD2}]
\label{lem-logbmolo}
Let $\Omega$ be an $\ed$-domain and $0 < \lambda \leq \led:= \frac{\epsilon^2 \delta}{320 n(1 + \sqrt{n} \epsilon)}$.  If $f \in \bmolo$ and $Q \subset \Omega$ is a cube with sidelength $\ell(Q) < \lambda$, then 
\begin{equation}
\label{eq-logbmolo}
|f|_Q  \lesssim \left(1 + \log\left(\frac{\lambda}{\ell(Q)}\right)\right) \|f\|_{\bmolo}.
\end{equation}
\end{lem}

If $\lambda_1 < \lambda_2 \leq \led$, then \eqref{eq-logbmolo} gives
\begin{equation}
\label{eq-normequiv}
\|f\|_\bmoltwo \leq 2\|f\|_\bmolone \leq C \log \frac{\lambda_2}{\lambda_1}\|f\|_\bmoltwo.
\end{equation}
This justifies using the notation $\bmoO$ for $\bmolo$, where we fix $\lambda = \led$ for the norm. The same estimates also show that the notions of vanishing at infinity and vanishing at the boundary are independent of the choice of $\lambda$.  Moreover, $\bmoloc$ can be denoted by $\bmoco$ independently of $\lambda$.

With the help of the three lemmas, we can prove an analogue of Proposition~\ref{prop-bmo_0}.
\begin{prop}
\label{prop-bmolo_0}
Let $\Omega$ be an $\ed$-domain and $f \in \bmoO$.  Then the following are equivalent:
\begin{enumerate}
\item[(i)] there exists a sequence $\{f_j\}$ converging to $f$  in  $\bmoO$, with $\dist(\supp(f_j),\bOmega)) > 0$ for each $j$;
\item[(ii)] $f$ vanishes at the boundary;
\end{enumerate}
Moreover, we have the equivalence of the following two conditions:
\begin{enumerate}
\item[(a)] $f$ is in the closure of $\bmoco$  in $\bmoO$;
\item[(b)] $f$ vanishes at the boundary and vanishes at infinity.
\end{enumerate}
\end{prop}

\begin{proof}
As was pointed out above, the implication in one direction of each pair holds in any domain, so it remains to show the implications (ii) $\implies$ (i) and (b) $\implies$ (a).

Since $\Omega$ is an $\ed$ domain and  $f \in \bmoO$, by Lemma~\ref{lem-logbmolo},  both \eqref{eq-1/t} and \eqref{eq-varphi} 
with
$$\varphi_\lambda(t) = 1 + \log_+\frac \lambda {4t}, \quad \lambda = \led,$$
hold for $f$ and also for $|f|$.
Assuming
$f$ vanishes at the boundary, we obtain from Lemma~\ref{lem-vanishing_at_boundary} the approximation by functions $f_j = fh_j$ in $\bmoO$ which are supported away from the boundary. 

If in addition $f$ vanishes at infinity, we apply the proof of Lemma~\ref{lem-vanishing_at_infinity} to $f_j$ and obtain the sequence of functions $f_{j,k} = f_j\psi_k = f h_j \psi_k$ having compact support in $\Omega$.  We need to show that by choosing $j$ and $k$ sufficiently large, we can make $\|f_{j,k} - f\|_\bmoO$ small.  This means estimating $\|f_{j,k} - f_j\|_\bmoO$.

From the proof of  Lemma~\ref{lem-vanishing_at_boundary}  we know that $0 \leq h_j \leq 1$ so for every $Q \subset \Omega$, $|f_j|_Q \leq |f|_Q$.  As $f$ vanishes at infinity, this means that \eqref{eq-largecubes-Omega} holds with $f$ replaced by $f_j$. Thus it suffices to bound the oscillation of $f_{j,k} - f_j$ on cubes $Q \subset \Omega$ with $\dist(Q, 0) \geq k/2$ and $\ell(Q) < \lambda$.  From \eqref{eq-Leibniz}, we can control this oscillation by
$$\sup_{Q \subset \Omega, \dist(Q, 0) \geq k/2, \ell(Q)< \lambda} \left(\fint_Q |f_j - (f_j)_Q|  +|(f_j)_Q|\ell(Q)k^{-1}\right),$$
which in turn, by \eqref{eq-1/t}, \eqref{eq-varphi}, \eqref{eqn-split} and \eqref{eq-lmo}, can be controlled by
$$ \sup_{Q \subset \Omega, \dist(Q, 0) \geq k/2, \ell(Q)< \lambda} \fint_Q |f - f_Q|  + j^{-1} + k^{-1}.
$$
It is then possible to choose $j$ and $k$ sufficient large to make both this quantity and $\|f_{j} - f\|_\bmoO$ small.  
\end{proof}

\subsection{Approximation in $\vmolo$}
\label{sec-approx-vmo}
We now consider functions in $\bmolo$ which also satisfy a Sarason-type VMO condition, namely
\begin{equation}
\label{eq-VMO-Omega}
\lim_{t \ra 0^+} \omegaO(f, t) = 0,
\end{equation}
where $\omegaO(f,t)$ is defined in \eqref{eqn-modulus-Omega}.  

The following two lemmas will lead to the proof of Theorem~\ref{thm-approxdomain}.
\begin{lem}
\label{lem-bounded}
Let $\Omega$ be a domain, $\lambda > 0$.  If $f$ is a bounded function satisfying \eqref{eq-VMO-Omega} then $f$
can be approximated in $\bmolo$ by bounded Lipschitz functions supported away from $\bOmega$.
\end{lem}

\begin{proof} This is the special case discussed after the proof of Lemma~\ref{lem-vanishing_at_boundary}, where the lemma can be applied with $\varphi_\lambda(t) = 1$ and tells us that we may approximate $f$ in $\bmolo$ by functions supported away from $\bOmega$.  Since the approximations are products of $f$ with bounded functions $h_j$, they are bounded.  To further approximate such functions by Lipschitz functions supported away from $\bOmega$ can be accomplished by the averaging process described at the beginning of Section~\ref{sec-approximation}, and whose details can be found in Section 2 of \cite{BD1}.  As the functions are supported away from $\bOmega$, the grid can be chosen fine enough so that the resulting Lipschitz functions are also supported away from the boundary.  Moreover, while the approximation in \cite{BD1} is in the $\BMO$ norm, it can be seen from the calculations following \cite[Equation (9)]{BD1}  that for large cubes, what is estimated are the averages of the error function, so the approximation is actually in the $\bmolo$ norm.  This argument is also referred to in the proof of \cite[Proposition 3]{BD1}. 
However, unlike in that case where $f$ is assumed to have compact support, or the domain is bounded, as in \cite{BN2}, here we cannot immediately obtain smooth functions with uniform bounds by convolving with a smooth mollifier, since we do not necessarily have control of the $L^1$ norm of $f$.  
\end{proof}

Based on the previous lemma, we will be able to prove Theorem~\ref{thm-approxdomain} if we can approximate $f$ in the $\bmolo$ norm by bounded functions.  The standard way to do this is by truncations, namely setting  $f^t = \max(\min(f, t), -t)$ and letting $t \ra \infty$.  In the case of a bounded domain, or when $f$ has compact support, the convergence of $f^t$ to $f$ in $L^1$ together with \eqref{eq-VMO-Omega} gives the convergence in $\BMO(\Omega)$ (see \cite[Lemma A1.4]{BN2}) and also in $\bmolo$.

For unbounded domains and general $f \in \bmolo$,
we need to find a criterion for approximation by bounded functions.  This is accomplished by the following lemma, which assumes some uniform control of the averages of $f$ on cubes of a given size.
Since we are assuming  \eqref{eq-VMO-Omega}, we can get by with a weaker assumption than those in Lemmas~\ref{lem-vanishing_at_infinity} and~\ref{lem-vanishing_at_boundary}.

\begin{lem}
\label{lem-vmo_equivalence}
Let $\Omega$ be a domain, $\lambda > 0$ and $f \in \bmolo$.   Suppose $f$ satisfies \eqref {eq-VMO-Omega} and for every $\ell > 0$, 
$$C_\ell = \sup_{2Q \subset \Omega, \ell(Q) \geq \ell} |f|_Q < \infty.$$
Then $f$
can be approximated in $\bmolo$ by bounded functions.
\end{lem}

\begin{proof}
  For a positive $\ell < \lambda/8\sqrt{n}$ and $x \in \Omega$, let $\ell(x) = \min(\frac{\dom(x)}{2\sqrt{n}}, \ell)$, take $Q_x$ to be the cube centered at $x$ with sidelength $\ell(x)$, and set
$$\tilde{f}(x) = \fint_{Q_x} f, \quad g(x) = \max(\min(\tilde{f}(x), C_\ell), -C_\ell).$$
As $g$ is a truncation, $|g| \leq C_\ell$.
Since $2Q_x \subset \Omega$, by the definition of $\ell(x)$ we have
$|\tilde{f}(x)|\leq C_\ell$
whenever $\dom(x) \geq 2\ell\sqrt{n}$, i.e.\ when  $x \in \Omegadel$, so $g = \tilde{f}$ on  $\Omegadel$.

To estimate the $\bmolo$ norm of $h :=  f - g$, we again use the local-to-global property (\cite[Lemma 3.5]{BD2}):
\begin{equation}
\label{eq-localglobal}
\|h\|_\bmolo \lesssim  \sup_{2Q \subset \Omega} \fint_B |h - h_Q|  + \sup_{2Q \subset \Omega, \ell(Q) \geq \lambda/2} |h|_Q.
\end{equation}
Note that cubes over which the supremum in the second term on the right is taken satisfy $\dist(Q, \bOmega) \geq \ell(Q)/2 \geq \lambda/4$, so to bound this term it suffices to control the averages $|h|_Q$ for $Q \subset \Omegalo$.  By the choice of $\ell$, $\Omegalo \subset \Omegadel$.

So let us first consider a cube $Q\subset \Omegadel$.  For $x \in Q$,
$g(x) =  \tilde{f}(x) = f_{Q_x}$ with $\ell(Q_x)= \ell$.  Suppose $\ell(Q) \geq \ell$.  Then for $z \in Q$,
\begin{equation}
\label{eq-Vdelta}
2^{-n}\ell^n \leq |Q_z \cap Q| \leq \ell^n, 
\end{equation}
and therefore
$$\fint_{Q_z \cap Q} |h| \leq \frac{2^n}{\ell^n}\int_{Q_z \cap Q} |f(x) -  f_{Q_x}| dx \leq \frac{2^n}{\ell^{2n}}\int_{Q_z} \int_{Q_x} |f(x) - f(y)| dy dx
\leq 2^{1 + 3n} \fint_{2Q_z} |f - f_{2Q_z}|,$$
where we have used the fact that $x \in Q_z \implies Q_x \subset 2Q_z$.
 Since $2Q_z\subset \Omega$, the right-hand-side is bounded by $2^{1 + 3n}\omegaO(f,2\ell)$.
From this, applying \eqref{eq-Vdelta} again, we get that
$$
 \fint_{Q}  |h|
  \leq  \fint_{Q}  \left\{\int_{z \in Q_x \cap Q} \frac{2^n dz}{|Q_z \cap Q|}\right\} |h(x)| dx\
 = 2^n\fint_{Q}\fint_{x \in Q_z \cap Q} |h(x)| dx dz \lesssim \omegaO(f,2\ell).
 $$
Thus $\omegaO(f,2\ell)$ controls the second term in \eqref{eq-localglobal}, as well as bounding the oscillation of $h$ over cubes in $\Omegadel$ with $\ell(Q) \geq \ell$.

It now remains to estimate the oscillation of $h$ over cubes $Q$ with $2Q \subset \Omega$ and such that either $Q \not\subset \Omegadel$ or $\ell(Q) < \ell$.  In the first case we also have $\ell(Q) \leq 2\dist(Q,\bOmega) < 4\sqrt{n}\ell$ so we will deal with the two simultaneously.

  Since $ |h - h_Q|  \leq  |f - f_Q|  +  |g - g_Q|$ and truncation reduces oscillation, we have
$$\fint_Q |h - h_Q|  \leq   \omegaO(f,4\sqrt{n}\ell) + \fint_Q |\tilde{f}(x) - \tilde{f}_Q| dx \leq \omegaO(f,4\sqrt{n}\ell) +  \fint_Q \fint_Q  |f_{Q_x}  - f_{Q_y}| dy dx.$$
For $x \in Q$, we take $Q'_x$ to be the cube centered at $x$ of sidelength $d = \frac{\dist(Q,\bOmega)}{2\sqrt{n}}$, noting that
$$d = \min\Big(\frac{\dist(Q,\bOmega)}{2\sqrt{n}} ,\ell\Big) \leq \ell(x) \leq \frac{\dist(Q,\bOmega) + \diam(Q)}{2\sqrt{n}} \leq  c d, \quad c:= 1 + 2\sqrt{n}.$$
Then
$$ |f_{Q'_x} - f_{Q_z}| \leq \fint_{Q'_x}|f - f_{Q_x}| \leq c^n  \fint_{Q_x}|f - f_{Q_x}| \leq c^n \omega(f, 4\sqrt{n}\ell).$$
The triangle inequality then gives us
$$\fint_Q \fint_Q  |f_{Q_z}  - f_{Q_y}| dy dx \leq 2c^n\omega(f, 4\sqrt{n}\ell) + \fint_Q \fint_Q  |f_{Q'_x}  - f_{Q'_y}| dy dx.$$
Finally, writing the averages as convolution with $d^{-n}\chi_{[-d/2,d/2]^n}$, we can bound the second term on the right-hand-side by
$$ \fint_Q \fint_Q  \fint_{[-d/2,d/2]^n} |f(x - z) - f(y - z)| dz dy dx = \fint_{[-d/2,d/2]^n}\fint_{Q-z} \fint_{Q-z} |f(u) - f(v)| du dv dz.$$
This quantity is also bounded by  $\omega(f, 4\sqrt{n}\ell)$, since for $z \in [-d/2,d/2]^n$, the choice of $d$ gives $\dist(Q-z, \bOmega) \geq \dist(Q,\bOmega) - d/2 > 0$ so $Q - z \subset \Omega$.

In summary, for the given $\ell$, we have shown the existence of a bounded function $g$ such that $\|f - g\|_\bmolo \lesssim \omega(f, 4\sqrt{n}\ell)$, where the constants depend only on $n$.  By choice of $\ell$, we can make this as small as we like.   
\end{proof}

\begin{proof}[Proof of Theorem~\ref{thm-approxdomain}]
On any domain $\Omega$ and for any $\lambda > 0$, if $f$ can be approximated in $\bmolo$ by uniformly continuous functions, then
\eqref{eq-VMO-Omega} holds.  For the converse inclusions we show the strongest, namely that  any function in $\bmolo$, with $\lambda$ as in the hypotheses of the theorem, can be approximated by bounded Lipschitz functions supported away from the boundary.  This follows from Lemmas~\ref{lem-bounded} and \ref{lem-vmo_equivalence} by noting that for any $\ell \le \lambda$, $C_\ell \leq \|f\|_{\bmo_\ell(\Omega)} < \infty$.
\end{proof}

As a corollary of Theorem~\ref{thm-approxdomain}, Lemma~\ref{lem-logbmolo} and Proposition~\ref{prop-bmolo_0}, we get the following.

\begin{prop}
\label{prop-vmo-cmo}
Let $\Omega$ be an $\ed$-domain and $f \in \bmoO$.  Then $f$ can be approximated in $\bmoO$ by bounded Lipschitz functions supported away from $\bOmega$ if and only if
\eqref{eq-VMO-Omega} holds.

In addition, the following are equivalent:
\begin{enumerate}
\item[(i)]  $f$ is in $\cmo(\Omega)$, the closure of  $C_c(\Omega)$ in $\bmoO$;
\item[(ii)] $f$ is in the closure of  $C_c^\infty(\Omega)$ in $\bmoO$;
\item[(iii)]  $f$ vanishes at infinity and 
\eqref{eq-VMO-Omega} holds.
\end{enumerate}
\end{prop}

\begin{proof}
We only need to show the implication (iii) $\implies$ (ii), as convolution with a smooth mollifier gives the approximation, in the $L^\infty$ norm, of a function in $C_c(\Omega)$ by functions in  $C_c^\infty(\Omega)$.  If  $f$ vanishes at infinity and 
\eqref{eq-VMO-Omega} holds, then $f$ vanishes at the boundary and we can follow the proof of Proposition~\ref{prop-bmolo_0}, starting with the sequence  $\{f_j\}$ of bounded Lipschitz functions supported away from $\bOmega$ which approximates $f$, and multiplying by the cut-off functions $\psi_k$ to make them of compact support in $\Omega$.
\end{proof}

\section{Proof of Theorem~\ref{thm2}: extension from a locally uniform domain}
\label{sec-extension1}
In this section we assume $\Omega$ is an $\ed$ domain and fix $\lambda\leq \led$, where $\led:= \frac{\epsilon^2 \delta}{320 n(1 + \sqrt{n} \epsilon)}$.
As in \cite{BD2}, for $f\in \bmolo$ we define a function $\Tlam f:\Rn \to \R$ by 
\begin{equation}\label{Ext_operator_T}
\Tlam f (x) = \begin{cases}
	f(x) & \textup{if }  x\in \Omega; \\
	f_{Q^*} &\textup{if } x\in Q\in E': \ell(Q)\leq \lambda; \\
	0 & \textup{otherwise} .
	\end{cases}	
\end{equation}
Here $E'$ denotes the Whitney decomposition of $\Omega'$, the complement of $\Omegabar$ (see \cite[Section VI.1]{Stein} for the definition and properties of Whitney cubes), and we have fixed, for each $Q \in E'$ with $\ell(Q)\leq \lambda$, a choice of matching cube $Q^* \in E$ with $\ell(Q) \leq \ell(Q^*) \leq 4\ell(Q)$ and $\dist(Q,Q^*) \leq C\ell(Q)$.  The existence of such a cube is guaranteed by  \cite[Lemma 2.4]{Jones2},  since $\led \leq \epsilon \delta/(16 n)$.  

By  \cite[Lemma 2.3]{Jones2}, $\bOmega$ has measure zero, so it suffices to extend $f$ to $\Omega'$.
We now proceed to average the step function $\Tlam f$ on $\Omega'$, as in \cite{BD1}.
Set
$$ \tT f(x) = \left\{  \begin{array}{cc}
        f(x), &  x\in \Omega,\\
        A (\Tlam f)(x)  & x\in \Omega',
    \end{array}
    \right.
$$
where the averaging operator $A$, applied applied to  $\phi = \Tlam f$ on $\Omega'$, is defined by
$$A(\phi)(x) := \fint_{B(x,R(x))} \phi(y) dy, \quad x \in \Omega'.$$
Here   
\begin{equation}
\label{eqn-R}
R(x)  = c_n  \dom(x)
\end{equation}
 where for $x \in \Omega'$, with an abuse of notation we denote $\dom(x) := \dist(x, \bOmega)$, just like for points inside $\Omega$.  As in \cite{BD1}, we choose the constant $c_n$ sufficiently small so that for $x \in \Omega'$, the collection 
$$\cN(x): = \{Q \in E': Q \cap B(x,R(x)) \neq \varnothing\}$$  
consists exactly of the Whitney cube containing $x$ and the Whitney cubes adjacent to it, and moreover that at the center $x_Q$ of a cube $Q \in E'$, $A(\phi)(x_Q) = \phi(x_Q)$.

Having proved in \cite{BD2} the boundedness of the map $f \ra \Tlam f$ from $\bmolo$ to $\bmol(\Rn)$, we now note that it also applies to the map $f \ra \tT f$.  As in \cite{BD1}, this follows from the fact that the difference between  $\Tlam f$ and its averaging $\tT f = A (\Tlam f)$ on a Whitney cube in $\Omega'$ is bounded by a constant times $\|\Tlam f\|_{\BMO(\Rn)}$, since by the choice of $c_n$, the averaging takes places only on adjacent Whitney cubes.
 
Having part (i) of Theorem~\ref{thm2}, we proceed to part (iv), which will give us the other two parts by approximation.

 \subsection{Boundedness on $\Lip_b(\Omega)$}
 The ideas are similar to those in Section 3.3 of \cite{BD1}, with two main differences.  First, we are dealing with bounded functions, which makes things easier.  On the other hand, we need to deal with the ``jump" in the definition of $\Tlam$, namely the separate definitions for cubes near and far away from the boundary.
 
 Starting with a bounded Lipschitz function $f$ on $\Omega$, by definition $\|\tT f\|_\infty = \|f\|_\infty$.  If we can show that $\tT f$ is Lipschitz on $\Omega'$ with $\|\tT f\|_{\Lip_b(\Omega')} \lesssim \|f\|_{\Lip_b(\Omega)}$,
 then the same argument as in Section 3.3 of \cite{BD1}, i.e.\ extending to $\bOmega$ by uniform continuity, will give us the desired Lipschitz continuity on all of $\Rn$ and prove part (iv) of Theorem~\ref{thm2}.
 
As in \cite[Lemma 1 and Corollary 3]{BD1}, the key is the inherent local Lipschitz nature of the averaging itself, which gives, for $x_1, x_2 \in \Omega'$,
\begin{equation}
\label{eq-Tlam} |\tT f(x_1)-\tT f(x_2)|  \le  C \sup_{Q,Q' \in \cN(x_1) \cup \cN(x_2)} \left|(\Tlam f)_{Q}- (\Tlam f)_{Q'}\right|  \min\left(\frac{|x_1 - x_2|}{\min_{i=1,2} \dom(x_i)}, 1\right).
\end{equation}
 Thus if both $x_i$ are sufficiently far from $\bOmega$, say contained in Whitney cubes of sidelength at least $\lambda/40\sqrt{n}$, then
$$
 |\tT f(x_1)-\tT f(x_2)| \leq C_{\lambda,n} \|\Tlam f\|_\infty |x_1 - x_2| \leq C_{\lambda,n} \|f\|_\infty |x_1 - x_2|.
$$
On the other hand, if one of the points, say $x_1$, belongs to a Whitney cube of sidelength less than $\lambda/40\sqrt{n}$, while the other, $x_2$, belongs to a cube of sidelength greater than $\lambda/4$, then by the triangle inequality and the properties of Whitney cubes, 
$$|x_1 - x_2| \geq \dom(x_2) - \dom(x_1) \geq \lambda/4 - (5\sqrt{n}\lambda/40\sqrt{n}) = \lambda/8,$$
and the trivial $L^\infty$ bound (not using \eqref{eq-Tlam}) gives
 $$
 |\tT f(x_1)-\tT f(x_2)| \leq C_{\lambda} \|\tT f\|_\infty |x_1 - x_2|  = C_{\lambda} \| f\|_\infty |x_1 - x_2| .
$$

Thus it remains to consider the case when both points lie in Whitney cubes, say $Q_1, Q_2$, respectively, of sidelength at most $\lambda/4$.  
Recalling that the Whitney cubes in $\cN(x_i)$ are either $Q_i$ or adjacent to it, by the properties of adjacent Whitney cubes, any cubes $Q,Q'$ in $\cN(x_1) \cup \cN(x_2)$, have sidelength bounded by $\lambda$,  so 
$$
 |\tT f(x_1)-\tT f(x_2)|  \le  C\sup_{Q,Q' \in \cN(x_1) \cup \cN(x_2)} \left|f_{Q^*}- f_{(Q')^*}\right|  \min\left(\frac{|x_1 - x_2|}{\min_{i=1,2} \dom(x_i)}, 1\right).
$$

In the special case when $Q_1$ and $Q_2$ are adjacent (which includes the case $Q_1 = Q_2$), since $\lambda < \epsilon \delta / 16n$, we can apply  \cite[Lemma 2.8]{Jones2} to get that the shortest Whitney chain connecting $Q_1^*$ and $Q_2^*$ has length $m$ bounded by a constant.
By transitivity, this also applies to any two cubes $Q^*, (Q')^*$ which are matching to cubes $Q, Q' \in \cN(x_1) \cup \cN(x_2)$.
Along such a chain, the difference of averages of $f$ on two adjacent cubes is bounded by  $\|f\|_{\Lip(\Omega)}$ times the sum of the diameters of the two cubes.  Again by the properties of Whitney cubes, the largest cube along such a chain has diameter at most $4^m$ times the smallest.  This gives (see also \cite[Lemma 4]{BD1} for the general case of a function in  $\BMO(\Omega)$ on any domain $\Omega$) that
$$\left|f_{Q^*}- f_{(Q')^*}\right| \lesssim \|f\|_{\Lip(\Omega)}\ell(Q^*).$$
Since in this special case the sidelengths of $Q\in \cN(x_1) \cup \cN(x_2)$ and a matching cube $Q^*$ are comparable to $\ell(Q_i)$, which are in turn comparable to $\dom(x_i)$, we get that
$$
 |\tT f(x_1)-\tT f(x_2)|  \lesssim  \|f\|_{\Lip(\Omega)}|x_1 - x_2|.
$$

Finally, when $Q_1$ and $Q_2$ are not adjacent, we follow the argument in Section 3.3 of \cite{BD1}.  In this case $|x_1 - x_2|$ must be at least as large as the sidelength of the smallest Whitney cube adjacent to either $Q_1$ or $Q_2$, which means $|x_1 - x_2| \gtrsim \max_{i=1,2}\ell(Q_i)$, 
so the estimate is achieved by showing that
$$\left|f_{Q_1^*}- f_{Q_2^*}\right| \lesssim \|f\|_{\Lip(\Omega)}\max_{i=1,2}\ell(Q_i) \lesssim \|f\|_{\Lip(\Omega)}|x_1 - x_2|,$$
and noting that the left-hand-side is just the difference of values of $\tT f$ at the center points of the $Q_i$, by the choice of $c_n$ in \eqref{eqn-R}.  The comparison of the value of $\tT f(x_i)$ to the value of $\tT f$ at the center of $Q_i$ can be bound by the same quantity using the case of adjacent (in this case identical) cubes above.

It should be noted that while Lipschitz extensions do not need any restriction on the domain, for this particular extension we have had to use some of the geometric properties of $\ed$ domains, namely the ``reflection" lemmas  from \cite{Jones2}, in particular Lemmas 2.5 and 2.8.

\subsection{The $\vmo$ and $\cmo$ extensions}
We follow the argument in \cite[Sections 3.4 and 3.5]{BD1}.  Here Proposition~\ref{prop-vmo-cmo} gives us, in the case of $f \in \vmo(\Omega)$, the approximation by bounded Lipschitz functions in $\bmoO$, and part (ii) of Theorem~\ref{thm2} follows from parts (i), (iv) and Theorem~\ref{thm-bourdaud}.

For $\cmo(\Omega)$, we just need to add to this argument the fact that the extension maps functions of compact support in $\Omega$ to functions of compact support in $\Rn$.  This is shown in \cite[Section 3.5]{BD1} in the homogeneous case (where compact support means the function is constant outside a compact set), but the case here is simpler since the extension $\Tlam f$ is zero, by definition, on Whitney cubes sufficiently far from $\bOmega$ and therefore, by the local nature of the averaging process (i.e.\ the choice of $c_n$ in \eqref{eqn-R}), so is $\tT f$.  More specifically, it will vanish on any $Q \in E'$ all of whose neighbors have sidelength greater than $\lambda$, which means it is guaranteed to vanish if $\ell(Q) > 4\lambda$.  

Suppose $f \in \bmoO$ is supported in $B(0,R)$.   By the argument above, we only need to study $\tT f$ on Whitney cubes $Q \in E'$ with $\ell(Q) \le 4\lambda$. Since $4\lambda < \epsilon \delta/16n$, we can apply \cite[Lemmas 2.4 and 2.5]{Jones2} to such cubes to conclude that their matching cubes $Q^*$ must lie within a distance comparable to $\ell(Q)$.  Thus $\Tlam f = 0$ on $Q$ whenever $\dist(Q, 0) \ge R + C_{\epsilon, n, \lambda}$.  Again by the local nature of the averaging process, this gives $\tT f = 0$ on any Whitney cube in $\Omega'$ such that all its neighbors lie in the complement of $B(0,R + C_{\epsilon, n, \lambda})$.  Since we are only looking at cubes of sidelength bounded by $4\lambda$, this shows that $\tT f$ is supported in $B(0, R')$ for some $R' > R$.

\section{Examples}
\label{sec-examples}   

We close the paper with a couple of examples to illustrate Theorem~\ref{thm-approxdomain} and some of the results in Section~\ref{sec-approximation2}.  

\begin{example}
\label{example1} \textnormal{In $\R^2$, our domain $\Omega$ consists of the left half-plane $\{(x,y): x < 0\}$ connected to infinitely many disjoint strips lying in the right half-plane, parallel to the positive $x$ axis. 
The $n$th strip $S_n$ has vertical width $1/n$ and horizontal length $L_n$.  If we take a cube of sidelength $\ell$, there are only finitely many strips $S_n$ which can contain this cube, namely those with $n < \ell^{-1}$. }

\textnormal{If $f$ is in $\bmolo$ and $\ell < 1/n < \lambda$, then starting from a cube $Q$ of sidelength $\ell$ in $S_n$, to reach a cube of sidelength $\lambda$ we have to get to the left half-plane, which means going through a chain of Whitney cubes in $S_n$ of sidelength $\gtrsim \ell$ and hence  whose length $m \lesssim L_n\ell^{-1}$.  The standard telescoping sum argument then gives the bound $|f|_Q \lesssim L_n\ell^{-1}\|f\|_\bmolo$. Since $n$ is bounded  by $\ell^{-1}$, we get that for $0 < \lambda' < \lambda$,
$$\|f\|_{\bmo_{\lambda'}(\Omega)}\lesssim \|f\|_\bmolo + \sup\{|f|_Q: \lambda' \leq \ell(Q) < \lambda\}
\lesssim \|f\|_\bmolo \Big(1 + \max_{n < \frac{1}{\lambda'}}\frac{L_n}{\lambda'}\Big) < \infty.$$
By Theorem~\ref{thm-approxdomain}, any $f \in \bmolo$ which satisfies the vanishing mean oscillation condition $\displaystyle{\lim_{t \ra 0^+} \omegaO(f, t) = 0}$ can therefore be approximated in $\bmolo$ by bounded Lipschitz functions.
}

\textnormal{On the other hand, such a domain $\Omega$ can only be an $\ed$ domain if the lengths $L_n$ go to zero sufficiently fast as $n \ra \infty$, in order to guarantee that when two points are at some distance $\delta' \leq \delta$, they can belong to $S_n$ only for sufficiently small $n$, depending on $\delta'$, and such $S_n$ are wide enough for them to be joined by the appropriate $\epsilon$ cigar.   If $L_n$ do not go to zero,  $\Omega$  will not be $\ed$.}

\textnormal{How fast do $L_n$ have to go to zero for $\Omega$ to be $\ed$?  Consider the function $f(x,y)$ on $\Omega$ which is $0$ on the left half-plane and is equal to $nx$ on $S_n$.  Since  all cubes contained in the strip $S_n$ are of sidelength at most $1/n$, the mean oscillation of $f$ is bounded by $1$.  Moreover, since only finitely many of the strips contain cubes of sidelength at least $\lambda$, and $0 \leq f \leq n L_n$ on $S_n$, we have 
$$\sup\{|f|_Q: \ell(Q) \geq \lambda\} \lesssim \max_{n < \frac{1}{\lambda}}n L_n < \infty.$$
If we could extend $f$ to a function in $\bmo(\Rn)$ then its averages would have to satisfy the logarithmic estimate \eqref{eq-logbmo}, so for each $n$, taking a cube $Q$  of sidelength $\frac{1}{2n}$ at the rightmost tip of $S_n$, we would need to have
$$n L_n \approx |f|_Q \lesssim \|f\|_{\bmolo} \log\Big(\frac{2}{\ell(Q)}\Big) =  \|f\|_{\bmolo} \log(4n).$$
Thus $L_n$ must be ${\mathcal O}(\frac{\log n}{n})$ as $n \ra \infty$.}
\end{example}

\begin{example}
\label{example2} \textnormal{We now give a variation on Example~\ref{example1} which does not satisfy the hypotheses of Theorem~\ref{thm-approxdomain}. For each integer $n \geq 1$, instead of one strip we attach to the left half-plane $n$ disjoint strips  in the right half-plane, $S_{n,j}, 1 \leq j \leq n$, with vertical width $\frac{1}{j}$ and horizontal length $n$.  This means that for any $\ell < 1$, there will be cubes of size $\ell$ in infinitely many strips, further and further away from the left half-plane.}

\textnormal{For this domain, fixing any $\lambda > 0$, we want to show there is a function $f$ in $\bmolo$ which is not in $\bmo_{\lambda'}(\Omega)$ for some $\lambda' < \lambda$.
Define $f$ to be $0$ in the left half-plane and on any strip which contains cubes of size $\lambda$.  On all the other strips $S_{n,j}$, define $f(x,y) = c_j x$, $c_j \neq 0$.  Then $f$ is Lipschitz with constant $c_j$ on $S_{n,j}$, which contains only cubes of sidelength bounded by $\frac{1}{j}$,  and therefore has vanishing mean oscillation provided $\frac{c_j}{j} \ra 0$ as $j \ra \infty$.  It is also in $\bmolo$ because it is zero on cubes of sidelength $\lambda$ or greater contained in $\Omega$.  However, for $\lambda' < \frac{1}{j} < \lambda$, there will be cubes $Q$ of size $\lambda'$ in all $S_{n,j}$ with $n \geq j$, for which the averages $|f|_Q \approx n c_j$.  Thus $f \not\in \bmo_{\lambda'}(\Omega)$. } 

\textnormal{Note that while $f$ satisfies the vanishing mean oscillation condition $\displaystyle{\lim_{t \ra 0^+} \omegaO(f, t) = 0}$, and its averages over all cubes of side $\lambda$ or larger are zero, it does not vanish at infinity in the sense Definition~\ref{def-vanishing_at}, as can be seen by looking at cubes of sidelength approximately $1$ in $S_{n,1}$, for arbitrary large $n$, whose distance from the origin is going to infinity but on which the oscillation of $f$ is approximately $c_1$.  Thus the implication  $2 \implies 3$ in Proposition~\ref{prop-cmo} does not hold in a general domain.}
\end{example}

\bibliographystyle{amsplain}

\end{document}